\documentclass[a4paper,12pt]{amsart}
\usepackage[cp1250]{inputenc}
\usepackage[T1]{fontenc}
\usepackage{amsthm}
\usepackage[leqno]{amsmath}
\usepackage{amssymb}
\usepackage{amsfonts}
\usepackage{enumerate}

\newcommand{\rr}{\mathbf{\mathbb{R}}}

\newcommand{\nn}{\mathbf{\mathbb{N}}}

\newcommand{\vf}{\varphi}

\newcommand{\dist}{\operatorname{dist}}

\newcommand{\Go}{\operatorname{\varrho_0}}
\newtheorem{lemma}{Lemma}
\newtheorem{theorem}{Theorem}

\newtheorem{cor}{Corollary}
\newtheorem{proposition}{Proposition}
\newtheorem{st}{Step}

\theoremstyle{remark}

\title{Local $C^r$-right equivalence of $C^{r+1}$ functions}

\author{Piotr Migus}
\thanks{This research was partially supported by NCN, grant 2012/07/B/ST1/03293.}
\address{Faculty of Mathematics and Computer Science, Wydzia\l{} Matematyki i Informatyki, Uniwersytet \L{}\'o{}dzki, Banacha 22, 90-238 \L{}\'o{}d\'z{}, Poland}
\date{\today}
\email{migus@math.uni.lodz.pl}

\begin{document}
\begin{abstract}
Let $f,g:(\rr^n,0)\rightarrow (\rr,0)$ be $C^{r+1}$ functions, $r\in \nn$. We will show that if $\nabla f(0)=0$ and there exist a neigbourhood $U$ of $0\in \rr^n$ and some constant $C>0$ such that 
$$
\left|\partial^m(g-f)(x)\right|\leq C \left|\nabla f(x)\right|^{r+2-|m|}, \quad x\in U,
$$
for any $m\in \nn_0^n$ such that $|m|\leq r$, then there exists a $C^r$ diffeomorphism $\vf:(\rr^n,0)\rightarrow (\rr^n,0)$ such that $f=g\circ \vf$ in a neighbourhood of $0$.
\end{abstract}

\maketitle

\section{Introduction}
Let $f,g:(\rr^n,0)\rightarrow (\rr,0)$. We say that $f$ and $g$ are $C^r$-\emph{right equivalent} if there exists a $C^r$ diffeomorphism $\vf:(\rr^n,0)\rightarrow (\rr^n,0)$  such that $f=g \circ \vf$ in a neighbourhood of $0$. Let $\nn$ denote the set of positive integers and $\nn_0=\nn\cup\{0\}$. A norm in $\rr^n$ we denote by $|\cdot|$ and by $\dist(x,V)$ - the distance of a point $x\in \rr^n$ to a set $V\subset\rr^n$. By $C^k(n)$, where $k,n\in \nn$, we denote the set of $C^k$ functions $(\rr^n,0)\rightarrow \rr$. Let $\mathcal{J}_fC^k(n)$ be the ideal in $C^k(n)$ generated by $\frac{\partial f}{\partial x_1},\dots, \frac{\partial f}{\partial x_n}$. The ideal $\mathcal{J}_fC^k(n)$ is called \emph{the Jacobi ideal in $C^{k}(n)$} (we will call it in short \emph{the Jacobi ideal}). 

In this paper we address the question under what conditions $C^r$-right equivalence of $C^{r+1}$ functions holds. There exists result which deals with $C^r$-right equivalence of $C^{r+2}$, namely J.~Bochnak has used Tougeron's Implicit Theorem to proved the following theorem (\cite[Theorem 1]{boch})

\noindent \emph{Let $f,g:(\rr^n,0)\rightarrow (\rr,0)$ be $C^{r+2}$ functions such that $\nabla f(0)=0$, $r \in \nn$. If $(g-f)\in \mathfrak{m}(\mathcal{J}_fC^{r+1}(n))^{2}$ then $f$ and $g$ are $C^r$-right equivalent. By $\mathfrak{m}$ we mean maximal ideal in the set of $C^{r+1}$ functions $(\rr^n,0)\rightarrow \rr$. }

Results presented in this paper are proven in the classical spirit of Kuiper-Kuo Theorem which deals with $C^0$-right equivalence of $C^r$ functions with isolated singularity at $0\in \rr^n$ (\cite{Kui}, \cite{Kuo} see also \cite{Sp}).  Moreover, in compare to Bochnak Theorem we assume geometric condition for $(g-f)$ instead algebraic condition. More precisely, we will prove the following

\begin{theorem}\label{the:2}
Let $f,g:(\rr^n,0)\rightarrow (\rr,0)$ be $C^{r+1}$ functions, $r\in \nn$. If $\nabla f(0)=0$ and there exist a neigbourhood $U$ of $\in \rr^n$ and some constant $C>0$ such that 
\begin{equation}\label{fun.zal}
\left|\partial^m(g-f)(x)\right|\leq C \left|\nabla f(x)\right|^{r+2-|m|}, \quad x\in U,
\end{equation}
for any $m\in \nn_0^n$ such that $|m|\leq r$, then $f$ and $g$ are $C^r$-right equivalent. 
\end{theorem}

After slight modification of the above theorem, we will obtain some sufficient condition for $C^0$-right equivalence (Theorem \ref{the:3}). Moreover, we will see that Theorem \ref{the:2} implies Theorem \ref{the:4}, where we assume that $(g-f)$ belongs to Jacobi ideal of $f$ to some power depends on $r$. It is worth mention about author's result (\cite{mgs}) where it has proved that  if two analytic functions $f,g$ are such that $(g-f)\in (f)^{r+2}$ then $f$ and $g$ are $C^r$-right equivalent, $(f)$ denote ideal generated by $f$. In this paper we will also prove that if two real analytic functions are $C^1$-right equivalent then they have the same exponent in the {\L}ojasiewicz gradient inequality (Proposition \ref{pro:2}).

\section{Auxiliary results}
Let us start this section from some obvious lemma. 
Let $M,m,k,r\in\nn$, $k\geq r$ and $M>r$. Moreover let $p,q_1,\dots,q_m\in C^k(n)$ and let $\mathcal{Q}C^{k}(n)$ denote the ideal in $C^k(n)$ generated by $q_1,\dots,q_m$. 

\begin{lemma}\label{lem:1}
If $p\in (\mathcal{Q}C^{k}(n))^M$ then 
\begin{itemize}
\item[(i)] $\frac{\partial^r p}{\partial x_{i_1} \dots \partial x_{i_r}}\in (\mathcal{Q}C^{k-r}(n))^{M-r}$ for $i_1,\dots,i_r \in \{1,\dots,n\}$,
\item[(ii)] $|p(x)|\leq  C |(q_1(x),\dots,q_n(x))|^M$ in a neighbourhood of $0$ and for some positive constant $C$.
\end{itemize}
\end{lemma}

From the above we obtain at once

\begin{cor}\label{wn:lem:1}
Let $f,g:(\rr^n,0)\rightarrow (\rr,0)$ be $C^{r+1}$ functions, $r \in \nn$. If $(g-f)\in (\mathcal{J}_fC^{r}(n))^{r+2}$, then there exists a neighbourhood $U$ of $0\in \rr^n$ of $0\in\rr^n$ and a constant $C>0$ such that 
$$
\left|\partial^m(g-f)(x)\right|\leq C\left|\nabla f(x)\right|^{r+2-|m|}, \quad x\in U.
$$
for any $m\in \nn_0^n$ such that $|m|\leq r$.
\end{cor}

The next two lemmas come from \cite{mgs} (respectively Lemma 2 and Lemma 3).

\begin{lemma}\label{lem:2}
Let $f:(\rr^n,0)\rightarrow(\rr,0)$ be a locally lipschitzian. Then there exist a neighbourhood $U$ at $0\in \rr^n$, constant $C>0$ such that for any $x\in U$, $\left| f(x)\right| \leq C \dist (x,V_f) $ ($V_f$ denote zero set of $f$).
\end{lemma}

\begin{lemma}\label{lem:3}
Let $\xi,\eta:U\rightarrow \rr$ be $C^{|k|}$ functions, $k\in\nn_0^n$, such that
$$
A_1|\eta(x)|^2 \leq|\xi(x)|\leq A_2|\eta(x)|^2, \quad |\partial \xi(x)|\leq A_3 |\eta(x)|, \quad x\in U,
$$
where $A_1,A_2,A_3>0$ are a positive constants and $U\subset \rr^n$ is a neighbourhood of the origin,. Then
$$
\left|\partial^k \left( \frac{1}{\xi(x)}\right)\right|\leq B|\eta(x)|^{-|k|-2},\quad x \in U,
$$
for some constant $B>0$.
\end{lemma}

The last lemma in this section is slight modification of \cite[Lemma 1]{Sp}.

\begin{lemma}\label{lem:4}
Let $G \subset \mathbb{R}\times \mathbb{R}^n$ be an open set, $W:G\rightarrow \mathbb{R}^n$ be a continuous mapping and let $V \subset \mathbb{R}^n$ be a closed set. If a system 
\begin{equation}\label{eqlemmaunique}
\frac{dy}{dt}=W(t,y)
\end{equation}
has a global uniqueness of solutions property in $G\backslash (\mathbb{R}\times V)$ and if  
$$
\left|W(t,x) \right|\leq C \dist(x,V)\quad \textrm{ for } (t,x)\in U,
$$
for some constant $C>0$ and some neighbourhood $U \subset G$ of set $(\mathbb{R}\times V)\cap G$, then \eqref{eqlemmaunique} has a global uniqueness of solutions property in $G$.
\end{lemma}

\section{Proof of Theorem \ref{the:2}}\label{sub:2}

Let $Z$ be the zero set of $\nabla f$ and let $U\subset \rr^n$ be a neighbourhood of $0\in\rr^n$ such that $f$ and $g$ are well defined. By Lemma \ref{lem:2} there exists a positive constant $A$ such that
\begin{equation}\label{E:5}
|\nabla f(x)|\leq A \dist(x,Z)\quad\textrm{ for }x\in U.
\end{equation}
Define the function $F:\rr^{n}\times U\rightarrow \rr$ by the formula
$$
F(\xi,x)=f(x)+\xi(g-f)(x),
$$
obviously
$$
\nabla F(\xi,x)=\left((g-f)(x),\nabla f(x)+\xi\nabla (g-f)(x)\right).
$$
Let $G=\{(\xi,x)\in\rr\times U:|\xi|< \delta\}$ where $\delta \in \nn$, $\delta>2$. From the above, diminishing $U$ if necessary, we have that there exists a constant $C_1>0$ such that
\begin{equation}\label{E:7}
|\nabla f(x)|\leq C_1|\nabla F(\xi,x)|\quad\textrm{ for }(\xi,x)\in G. 
\end{equation}
Indeed,
$$
|\nabla F(\xi,x)|\geq |\nabla f(x)-\xi \nabla (g-f)(x)|\geq |\nabla f(x)|- |\xi||\nabla (g-f)(x)|.
$$
Since $r\geq 1$ then from \eqref{fun.zal} we get 
$$
|\nabla (g-f)(x)|\leq C_2|\nabla f(x)|^{r+1}\leq C_2|\nabla f(x)|^2
$$
for some positive constant $C_2$. Hence, diminishing $U$ if necessary, 
$$
|\nabla F(\xi,x)|\geq |\nabla f(x)|- |\xi|C_2|\nabla f(x)|^2\geq \frac{1}{C_1}|\nabla f(x)| \quad \textrm{ for } (\xi,x)\in G.
$$ 
Moreover, from definition of $\nabla F$ we get at once, that there exists a positive constant $C_3$ such that
\begin{equation}\label{E:7'}
|\nabla f(x)|\geq C_3|\nabla F(\xi,x)|\quad\textrm{ for }(\xi,x)\in G.
\end{equation}

Now we will show that the mapping $X:G\rightarrow \mathbb{R}^n\times \mathbb{R}$ defined by
$$
X(\xi,x)=(X_1,\dots,X_{n+1})= \left\{ \begin{array}{ll}
\frac{(g-f)(x)}{\left|\nabla F(\xi,x) \right|^2}\nabla F(\xi,x) & \textrm{ for $x\notin Z$}\\
0 & \textrm{ for $x\in Z$}
\end{array} \right.
$$
is a $C^r$ mapping. The proof of this fact will be divided into several steps.

\begin{st}\label{st:1}
The mapping $X$ is continuous in $G$.
\end{st}
Indeed, let $h_i(\xi,x)=(g-f)(x)\frac{\partial F}{\partial x_i}(\xi,x)$. Then from \eqref{fun.zal} and \eqref{E:7'} we have 
$$
|h_i(\xi,x)|=|\frac{\partial F}{\partial x_i}(\xi,x)||(g-f)(x)|\leq CC_1'|\nabla f(x)|^{r+3}, \quad (\xi,x)\in G \backslash (\rr \times Z).
$$
Moreover, from definition of $X$, \eqref{E:5} and \eqref{E:7} we obtain
\begin{equation}\label{E:8}
|X_i(\xi,x)|\leq CC_1'|\nabla f(x)|^{r+1}\leq A' \dist(x,Z)^{r+1}, \quad (\xi,x)\in G \backslash (\rr \times Z),
\end{equation}
for some constant $A'>0$. The above inequality also holds for $(\xi,x)\in G \cap (\rr \times Z)$, therefore $X$ is continuous in $G$.

\begin{st}\label{st:2}
Let $\alpha=(\alpha_0, \dots , \alpha_n) \in \nn_0^{n+1}$ be a multi-index such that $|\alpha|\leq r$, then, 
$$
|\partial^{\alpha} X_i (\xi,x)|\leq A''\dist(x,Z)^{r+1-|\alpha|} \textrm{ for } (\xi,x)\in G \backslash (\rr \times Z).
$$ 
where $\partial ^ {\alpha} X_i=\partial ^{\alpha_0}\cdots \partial^{\alpha_{n+1}}X_i=\frac{\partial^{|\alpha|}X_i}{\partial \xi^{\alpha_0}\partial x_1^{\alpha_1} \cdots \partial x_n^{\alpha_n}}.$
\end{st}
Indeed, let us take $(\xi,x)\in G \backslash (\rr \times Z)$ from Leibniz rule we have
\begin{equation}\label{E:9}
\partial^{\alpha} X_i(\xi,x)=\sum_{\beta \leq \alpha} \binom{\alpha}{\beta}\partial^{\alpha-\beta}(h_i(\xi,x))\partial^{\beta}\left( \frac{1}{|\nabla F(\xi,x)|^2}\right).
\end{equation}

Diminishing $G$ if necessary, from Lemma \ref{lem:3} we obtain
$$
\left|\partial^{\beta}\left( \frac{1}{|\nabla F(\xi,x)|^2}\right)\right|\leq\frac{A''_{\beta}}{|\nabla F(\xi,x)|^{|\beta|+2}},
$$
for some constants $A''_{\beta}>0$. Therefore from \eqref{E:9} we have
\begin{equation}\label{E:15}
|\partial^{\alpha}X_i(\xi,x)|\leq \sum_{\beta\leq \alpha}\binom{\alpha}{\beta}|\partial^{\alpha-\beta}(h_i(\xi,x))|\frac{A''_{\beta}}{|\nabla F(\xi,x)|^{2|\beta|+2}}.
\end{equation}
From \eqref{fun.zal} and \eqref{E:7'} we have 
\begin{equation}\label{E:14}
|\partial^{\alpha-\beta}(h_i(\xi,x))|\leq B_{\alpha-\beta} |\nabla f(x)|^{r+3-|\alpha|+|\beta|}
\end{equation}
for some positive constant $B_{\alpha-\beta}$. Finally from \eqref{E:15}, \eqref{E:14} \eqref{E:7} and \eqref{E:5} we obtain
\begin{align*}
|\partial^{\alpha} X_i(\xi,x)|&\leq \sum_{\beta \leq \alpha}\binom{\alpha}{\beta}B_{\alpha-\beta}|\nabla f(x)|^{r+3-|\alpha|+|\beta|}\frac{A''_{\beta}}{|\nabla F(\xi,x)|^{|\beta|+2}}\\
&\leq \sum_{\beta \leq \alpha}\binom{\alpha}{\beta}A''_{\beta}B_{\alpha-\beta}|\nabla f(x)|^{r+3-|\alpha|+|\beta|-|\beta|-2}\\
&\leq \frac{A''}{A} |\nabla f(x)|^{r+1-|\alpha|}\leq A'' \dist(x,Z)^{r+1-|\alpha|},
\end{align*}
for $(\xi,x)\in G \backslash (\rr \times Z)$ and for some constant $A''>0$.

\begin{st}\label{st:3}
Partial derivatives $\partial^{\alpha} X_i$ vanish for $(\xi,x)\in G \cap (\rr \times Z)$ and $|\alpha|\leq r$.
\end{st}

Indeed, we will  carry out induction with respect to $|\alpha|$. Let $t\in \rr$, $x \in Z$ and let $x^t_m=(x_1,\dots,x_m+t,\dots, x_n)$. For $|\alpha|=0$ hypothesis is obvious. Assume that hypothesis is true for $|\alpha|\leq r-1$. Then from Step \ref{st:2} we have
\begin{align*}
\frac{|\partial^{\alpha}X_i(\xi,x^t_m)-\partial^{\alpha}X_i(\xi,x)|}{|t|}&=\frac{|\partial^{\alpha}X_i(\xi,x^t_m)|}{|t|}\leq \frac{A''\dist(x^t_m,Z)^{r+1-|\alpha|}}{|t|}\\
&\leq \frac{A''|t|^{r+1-|\alpha|}}{|t|}=A''|t|^{r-|\alpha|}.
\end{align*}
Since $r-|\alpha|\geq r-r+1= 1$, we obtain $\partial^{\gamma}X_i(\xi,x)=0$ for $x\in Z$ and $|\gamma|=|\alpha|+1$. This completes Step \ref{st:3}.

In summary from Step \ref{st:1}, \ref{st:2} and \ref{st:3} we obtain that $X_i$ are $C^r$ functions in $G$. Therefore $X$ is a $C^r$ mapping in $G$.

Define a vector field $W:G\rightarrow \rr^n$ by the formula
\begin{equation}\label{pole1}
W(\xi,x)=\frac{1}{X_1(\xi,x)-1}(X_2(\xi,x),\dots,X_{n+1}(\xi,x)).
\end{equation}
Diminishing $U$ if necessary, we may assume that $A'\dist(x,Z)< \frac{1}{2}$. From \eqref{E:8} we obtain
\begin{equation}\label{pole2}
\left|X_1(\xi,x)-1\right|\geq 1-\left|X(\xi,x)\right|\geq 1-A'\dist(x,Z)>\frac{1}{2}, 
\end{equation}
for $(\xi,x)\in G$. Hence the field $W$ is well defined and it is a $C^r$ mapping. 

Consider the following system of ordinary differential equations
\begin{equation}\label{E:12}
\frac{dy}{dt}=W(t,y).
\end{equation}
Since $r\geq 1$, then $W$ is at least of class ${C}^1$ on $G$, so it has a uniqueness of solutions property in $G$. By solving \eqref{E:12} we obtain that a general solution is of $C^r$-class. Moreover, by definition of $G$ any solution is defined on interval $[0,1]$. Hence, there exists a $C^r$ diffeomorphism $\vf:(\rr^n,0)\rightarrow (\rr^n,0)$ given by formula $\vf(x)=y_x(1)$, where $y_x:[0,1]\rightarrow \rr^n$ is solution of system \eqref{E:12} with initial condition $y_x(0)=x$. 

Note that for any $x\in U$, 
\begin{equation}\label{E:13}
F(t,y_x(t))=const. \quad \textrm{ in }[0,1].
\end{equation}
Indeed, from definition of $W$ we derive the formula
$$
[1,W(\xi,x)]=\frac{1}{X_1(\xi,x)-1}(X(\xi,x)-e_1)\quad \textrm{ for }(\xi,x)\in G,
$$
where $e_1=[1,0,\dots,0]\in\mathbb{R}^{n+1}$ and $[1,W]:G\rightarrow \mathbb{R}\times \mathbb{R}^n$. Thus, if we denote by $\left\langle a,b\right\rangle$ the scalar product of two vectors $a$, $b$, then for $t\in [0,1]$, we have
\begin{align*}
&\frac{dF(t,y_x(t))}{dt}= \left\langle (\nabla F)(t,y_x(t)),[1,W(t,y_x(t))]\right\rangle\\
&=\frac{1}{X_1(t,y_x(t))-1}\left( \left\langle (\nabla_x F)(t,y_x(t)),X(t,y_x(t))\right\rangle -\frac{\partial F}{\partial \xi}(t,y_x(t))\right)\\
&=\frac{1}{X_1(t,y_x(t))-1}(g(y_x(t))-f(y_x(t))-g(y_x(t))+f(y_x(t)))=0.
\end{align*}
This gives \eqref{E:13}. Finally, \eqref{E:13} yields
$$
f(x)=F(0,x)=F(0,y_x(0))=F(1,y_x(1))=F(1,\varphi(x))=g(\varphi(x)).
$$
for $x\in U$. This ends the proof.
 
$\hfill \Box$

\section{Additional results}\label{sub:3}

Under assumptions of Theorem \ref{the:2}, note that in the situation when $r=0$, we have that $\nabla f$ is a continuous mapping and we can't use Lemma \ref{lem:2}, so we should assume that $\nabla f$ is a locally lipschitzian mapping. Moreover, contition \eqref{fun.zal} has the form 
$$
\left|(g-f)(x)\right|\leq C\left|\nabla f(x)\right|^{2}, \quad x\in U,
$$
so inequalities \eqref{E:7} and \eqref{E:7'} are false. But when we will assume additionally that 
$$
\left| \nabla (g-f)\right| \leq \left|\nabla f(x)\right|^{2}, \quad x\in U,
$$
for some constant $C'>0$, then those inequalities will be already true. Moreover, from those inequalities we obtain that vector field \eqref{pole1} is continuous in $G$ and locally lipschitzian in $G\backslash (\mathbb{R}\times Z)$. Therefore system \eqref{E:12} has a global uniqueness of solutions property only in $G\backslash (\mathbb{R}\times Z)$. But from \eqref{pole2} and Lemma \ref{lem:4} we have that \eqref{E:12} has a global uniqueness of solutions property in $G$. Therefore, due to the above, we obtain the following sufficient condition for $C^0$-right equivalence. Additionally to obtain that mapping $\nabla F $ is locally lipschitzian we should assume that $\nabla g$ is locally lipschitzian. 

\begin{theorem}\label{the:3}
Let $f,g:(\rr^n,0)\rightarrow (\rr,0)$ be $C^{1}$ functions such that $\nabla f, \nabla g$ are locally lipschitzian mappings, $r\in \nn$. If $\nabla f(0)=0$ and there exists a neigbourhood $U$ of $0\in \rr^n$ and constants $C,C'>0$ such that 
$$
\left|(g-f)(x)\right|\leq C\left|\nabla f(x)\right|^{2},\quad \left| \nabla(g-f)(x)\right|\leq C'\left|\nabla f(x)\right|^{2} \quad x\in U,
$$
then $f$ and $g$ are $C^0$-right equivalent. 
\end{theorem}

From Theorem \ref{the:2} and Corollary \ref{wn:lem:1} we obtain immediately
\begin{theorem}\label{the:4}
Let $f,g:(\rr^n,0)\rightarrow (\rr,0)$ be $C^{r+1}$ functions, $r\in \nn$. If $\nabla f(0)=0$ and $(g-f)\in (\mathcal{J}_fC^{r}(n))^{r+2}$ then $f$ and $g$ are $C^r$-right equivalent. 
\end{theorem}

It seems that Bochnak Theorem (\cite[Theorem 1]{boch}) is stronger than Theorem \ref{the:4}, because in the first theorem we assume that power of Jacobi ideal is constant, whereas in the last theorem this power is depend on $r$. But in the other hand in Theorem \ref{the:4} assumption about class of $f,g$ is weaker than in Bochnak Theorem. So it is difficult to say which theorem is better.

\section{{\L}ojasiewicz exponent in the gradient inequality}\label{sec:2}

Under the additional assumption of analyticity of functions, we will show that if two functions are $C^1$-right equivalent then their {\L}ojasiewicz exponents in the gradient inequality are the same.

Let $f:(\rr^n,0)\rightarrow(\rr,0)$ be an analytic function. It is known that there exists a neighbourhood of $0\in \rr^n$ and constants $C,\eta>0$ such that the following \emph{{\L}ojasiewicz gradient inequality} holds
$$
|\nabla f(x)|\geq C|f(x)|^{\eta}, \quad \textrm{ for }x\in U.
$$
The smallest exponent $\eta$ in the above inequality is called \emph{{\L}ojasiewicz exponent in the gradient inequality} and is denoted by $\Go(f)$ (cf. \cite{Loj2}, \cite{Loj3}).

\begin{proposition}\label{pro:2}
Let $f,g:(\rr^n,0)\rightarrow(\rr,0)$ be analytic functions. If $f$ and $g$ are $C^1$-right equivalent then $\Go(f)=\Go(g)$.
\end{proposition}

\begin{proof}
By the assumption there exists a $C^1$ diffeomorphism \linebreak $\vf:(\rr^n,0)\rightarrow (\rr^n,0)$ such that $g=f\circ \vf$ and $f=g\circ \vf^{-1}$. Moreover there exists a neighbourhood $U$ of $0\in \rr^n$ and a constant $C>0$ such that
$$
|\nabla g(x)|\geq C|g(x)|^{\Go(g)}, \quad \textrm{ for }x\in U.
$$
By $J(\vf)$ we denote the Jacobian matrix of mapping $\vf$ and by $\left\|J(\vf)\right\|$ the norm of this matrix. Note that, diminishing $U$ if necessary, there exists a constant $A>0$ such that  
$$
\left\|J(\vf(x))\right\|\leq A,  \quad \textrm{ for }x \in U.
$$
Moreover, $\nabla g= \nabla (f\circ \vf)=\nabla f(\vf)\cdot J(\vf)$ and from the above,
$$
|\nabla g(x)| \leq |\nabla f(\vf(x))|\left\|J(\vf(x))\right\|\leq A |\nabla f(\vf(x))| \quad \textrm{ for }x \in U.
$$
Hence, diminishing $U$ if necessary
$$
|\nabla f(\vf(x))| \geq \frac{1}{A} |\nabla g(x)|\geq \frac{C}{A}|g(x)|^{\Go(g)}= \frac{C}{A}|f(\vf(x))|^{\Go(g)} \;\; \textrm{ for }x \in U.
$$
Therefore $\Go(f)\geq \Go(g)$. Analogously we get $\Go(f)\leq \Go(g)$. Hence we have $\Go(f)=\Go(g)$.
\end{proof}

From Proposition \ref{pro:2} and Theorem \ref{the:4} we obtain

\begin{cor}\label{col:1}
Let $f,g:(\rr^n,0)\rightarrow(\rr,0)$ be analytic functions. If $\nabla f(0)=0$ and $(g-f)\in \mathcal{J}_f^3$ then $\Go(f)=\Go(g)$, where $J_f$ denotes the Jacobi ideal of $f$ in the ring of germs of analytic functions $(\rr^n,0)\rightarrow \rr$.
\end{cor}

\subsection*{Acknowledgements}
I am deeply grateful to Stanis{\l}aw Spodzieja for his valuable comments and advices.

\end{document}